\documentclass{article}

\usepackage{cite}
\usepackage[pdftex]{graphicx}
\usepackage{amsmath}
\usepackage{amsfonts}
\usepackage{amssymb}
\usepackage{amsthm}
\usepackage{color}
\usepackage[usenames,dvipsnames,svgnames,table]{xcolor}
\usepackage{algorithmic}
\usepackage{algorithm}
\usepackage{array}
\usepackage{url}
\usepackage{float}
\usepackage{graphicx}
\usepackage{amsthm,amssymb,latexsym,graphics,amscd,amsmath,amscd,enumerate,color}
\usepackage[T1]{fontenc}
\usepackage[english]{babel}
\usepackage{multicol}
\usepackage{multirow}
\usepackage{array}
\usepackage{booktabs}
\usepackage[usenames,dvipsnames,svgnames,table]{xcolor}
\usepackage{graphicx}
\usepackage{tikz}
\usetikzlibrary{matrix}
\usepackage{wrapfig}
\usepackage{graphics}
\usepackage{lmodern} \normalfont
\usepackage{mathrsfs}
\usepackage{latexsym}
\usepackage{yfonts}
\usepackage[all,cmtip]{xy}
\usepackage{mathtools}

\newcommand{\Ccl}{\mathbb{C}\ell}
\newcommand{\Cl}{\mathcal{C}\ell}

\newcommand{\mb}{\mathbb}

\newcommand{\Sp}{\mbox{Spin}}
\newcommand{\Spc}{\mbox{Spin}^{\mb{C}}}

\newtheorem{theorem}{Theorem}

\newtheorem{lemma}[theorem]{Lemma}

\title{Immersion in $\mathbb{R}^n$ by complex spinors}
\author{Rafael de Freitas Le\~ao\footnote{leao@ime.unicamp.br} and Samuel Augusto Wainer\footnote{samuelwainer@ime.unicamp.br}}

\begin{document}

\maketitle

\noindent \textbf{Keywords:} Immersion, Spinors, Killing Equation 

\vspace{1cm}

\noindent \textbf{Abstract} \\

A beautiful solution to the problem of isometric immersions in $\mathbb{R}^n$ using spinors was found by Bayard, Lawn and Roth \cite{bayard16}. However to use spinors one must assume that the manifold carries a $\Sp$-structure and, especially for complex manifolds where is more natural to consider $\Spc$-structures, this hypothesis is somewhat restrictive. In the present work we show how the above solution can be adapted to $\Spc$-structures.

\section{Introduction}

The problem of isometric immersions of Riemannian manifolds is a classical and widely studied problem in differential geometry. Since 1998, manly because of the work of Thomas Friedrich \cite{friedrich98}, this problem got a new understanding. In \cite{friedrich98}, since Riemannian 2-manifolds are naturally $\Sp$-manifolds, Friedrich showed that isometric immersions of these 2-manifolds are related with spinors that satisfies a Dirac type equation. Such relation can be understood as a spinorial approach of the standard Weierstrass representation.

Since then, a lot of work has been done to further understand this relation and to extend it to more general spaces than Riemannian 2-manifolds. Some remarkable examples of these contribuitions are, for example: in $2004$ Bertrand Morel \cite{morel04} extended Friedrich's spinorial representation of isometric immersions in $\mathbb{R}^{3}$ to $\mathbb{S}^{3}$ and $\mathbb{H}^{3}$; in $2008$ Marie-Amelie Lawn \cite{lawn08}
showed how a given Lorentzian surface $(M^{2},g)$ can be isometrically immersed in the pseudo-Riemannian space $\mathbb{R}^{2,1}$ using spinorial techniques; in $2010$ Lawn and Julien Roth \cite{lawnroth10} exhibit a spinorial characterization of Riemannian surfaces isometrically immersed in the 4-dimensional spaces $\mathbb{M}^4$, $\mathbb{M}^{3}\times \mathbb{R}$ $(\mathbb{M}\simeq (\mathbb{R},\mathbb{S},\mathbb{H})$; using the same spinorial techniques, Lawn and Roth \cite{lawnroth11} in $2011$, presented the spinorial characterization of isometric immersions of arbitrary dimension surfaces in $3$-dimensional space forms, thus generalizing Lawn's work in $\mathbb{R}^{2,1}$; in $2013$ Pierre Bayard \cite{bayard13} proved that an isometric immersion of a Riemannian surface $M^{2}$ in $4$-dimensional Minkowski space $\mathbb{R}^{1,3}$, with a given normal bundle $%
E $ and a given mean curvature vector $\vec{H}\in \Gamma (E)$, is equivalent to the existence of a normalized spinor field $\varphi \in \Gamma (\Sigma M \otimes \Sigma E)$ which is solution of the Dirac equation $D\varphi =\vec{H}\cdot \varphi $ in the surface.

More recently, Bayard, Lawn and Roth \cite{bayard16}, studied spinorial immersions of simply connected $\Sp$-manifolds of arbitrary dimension. The main idea is to use the regular left representation of the Clifford algebra on itself, given by left multiplication, to construct a $\Sp$-Clifford bundle of spinors. In this bundle, using the Clifford algebra structure, is possible to define a vector valued scalar product and, combining this product with a spinor field that satisfies a proper equation, define a vector valued closed 1-form whose integral gives a isometric immersion analogous to the Weierstrass representation of surface.

This work, \cite{bayard16}, provides a beautiful generalization of the previous work relating the Weierstrass representation to spinors. However, mainly when we are considering complex manifolds, the hypothesis of existence of a $\Sp$-structure is somewhat restrictive. Complex manifolds always have a canonical $\Spc$-structure that can be used to construct spinor bundles,  but the existence of a $\Sp$-structure is related to square roots of the canonical bundle and they do not always exist.

The aim of the present work is to show how the ideas of \cite{bayard16} can be generalized to spinor bundles associated to $\Spc$-structure, providing a more natural setting to complex manifolds. Precisely we prove:

\begin{theorem}
	Let $M$ a simply connected $n$-dimensional manifold, $E\rightarrow M$ a
	vector bundle of rank $m$, assume that $TM$ and $E$ are oriented and $Spin^{\mathbb{C}}.$ Suppose that $B:TM\times TM\rightarrow E$ is symmetric and
	bilinear. The following are equivalent:
	
	\begin{enumerate}
		\item There exist a section $\varphi \in \Gamma (N\sum\nolimits^{ad\mathbb{C}})$ such that 
		\begin{equation}
		\nabla _{X}^{\Sigma ^{ad\mathbb{C}}}\varphi =-\frac{1}{2}\sum_{i=1}e_{i}\cdot B(X,e_{i})\cdot \varphi +\frac{1}{2} \textbf{i}~A^{l}(X)\cdot \varphi ,~~\forall
		X\in TM.  
		\end{equation}
		
		\item There exist an isometric immersion $F:M\rightarrow \mathbb{R}^{\left(n+m\right) }$ with normal bundle $E$ and second fundamental form $B$.
	\end{enumerate}
	
	Furthermore, $F=\int \xi $ where $\xi $ is the $\mathbb{R}^{(n+m)}$-valued $%
	1 $-form defined by%
	\begin{equation}
	\xi (X):=\left\langle \left\langle X\cdot \varphi ,\varphi \right\rangle
	\right\rangle ,~~\forall X\in TM.  \label{form}
	\end{equation}
\end{theorem}

\section{Adapted Structures}

Let $E \rightarrow M$ be a hermitian vector bundle over $M$. A $\Spc$-structure on $E$ is defined by the following double covering

\begin{displaymath}
\xymatrixcolsep{1pc}\xymatrixrowsep{1pc}\xymatrix{ & Spin_{n}^{\mathbb{C}}
		\ar[rr]^{p^{\mathbb{C}}=\lambda^{\mathbb{C}} \times l^{\mathbb{C}}} \ar@{^{(}->}[dd] & & SO_n \times S^1 \ar@{^{(}->}[dd] \\ \mathbb{Z}_{2}
		\ar[ru] \ar[rd]& & &\\ & P_{Spin_{n}^{\mathbb{C}}}(E) \ar[rr]^{\Lambda^{\mathbb{C}}}
		\ar[rd]_{\pi^{\prime}} & & P_{SO_{n}}(E) \times_M P_{S^1}(E) \ar[ld]^{\pi} \\ & & M & }
\end{displaymath}
where $\Spc$ is the group defined by

\begin{displaymath}
	Spin_n^{\mathbb{C}} = \frac{Spin_n \times S^1}{ \{ (-1,-1) \} }, 
\end{displaymath}
and $S^1=U(1) \in \mathbb{C}$ is understood as the unitary complex numbers. As usual, a $\Spc$-structure can be viewed as a lift of the transition functions of $E$, $g_{ij}$, to the group $\Spc$, $\tilde{g}_{ij}$, but now the transition functions are classes of pairs $\tilde{g}_{ij} = \left[ (h_{ij},z_{ij}) \right]$, where $h_{ij}: U_i \cap U_j \rightarrow Spin_n$ and $z_{ij}: U_i \cap U_j \rightarrow S^1=U(1)$. 

The identity on $\Spc$ is the class $\left\{ (1,1), (-1,-1) \right\}$. Because of this, neither $h_{ij}$ or $z_{ij}$ must satisfy the cocycle condition, only the class of the pair. But, $z_{ij}^2$ satisfies the cocycle condition and defines a complex line bundle $L$, associated with the $P_{S^1}$ principal bundle in the above diagram, called the determinant of the $\Spc$-structure.

The description using transition functions is useful to make clear that $\Spc$-structures are more general than $\Sp$-structures. In fact, given a $\Sp$-structure $P_{\Sp}(E) \rightarrow P_{SO}(E)$ we immediately get a $\Spc$-structure by considering $z_{ij}=1$, in other words, by considering the trivial bundle as the determinant bundle of the structure. On the other hand \cite{hitchin}, a $\Spc$-structure produces a $\Sp$-structure iff the determinant bundle has a square root, that is, the functions $z_{ij}$ satisfies the cocycle condition. 


Another way where $\Spc$-structures are natural is  when we consider an almost complex manifold $(M,g,J)$. In this case the tangent bundle can be viewed as an $U(n)$ bundle, and the natural inclusion $U(n) \xhookrightarrow{} SO(2n)$ produces a canonical $\Spc$-structure on the tangent bundle \cite{friedrich00,nicolaescu}. For this canonical structure the determinant bundle is identified with $\wedge^{0,n} M$ and the spinor bundle constructed using an irreducible complex representation of $\Cl(2n)$ is isomorphic with $\wedge^{0,*}M = \oplus_{k=0}^n \wedge^{0,k}M$. So, various structures on spinors can be described using know structures of $M$.

Unlike the usual case for $\Sp$-structures, a metric connection on $E$ is not enough to produce a connection on $P_{\Spc}(E)$, for this, we also need a connection on the determinant bundle of the structure to get a connection on $P_{SO}(E) \times P_{S^1}(E)$ and be able to lift this connection to $P_{\Spc}(E)$.

To understand the problem of immersions using the Dirac equation in the case of $\Spc$-structures, and spinors associated to this structure, we need to understand adapted $\Spc$-structures on submanifolds. The difference to the standard $\Sp$ case is that we need to keep track of the determinant bundle. Using the ideas of \cite{bar98}, we can describe the adapted structure.

Consider a $\Spc$ $(n+m)$-dimensional manifold $Q$ and a isometrically immersed $n$-dimensional $\Spc$ submanifold $M \xhookrightarrow{} Q$. Let
\begin{displaymath}
	\begin{split}
		&P_{\Spc_{(n+m)}}(Q) \xrightarrow{\Lambda^Q} P_{SO_{(n+m)}}(Q) \times P_{S^1}(Q) \\
		& \left. P_{\Spc_{(n+m)}}(Q) \right|_M \xrightarrow{\Lambda^Q} \left. P_{SO_{(n+m)}}(Q) \bigg|_M \times P_{S^1}(Q) \right. \\
		&P_{\Spc_n}(M) \xrightarrow{\Lambda^M} P_{SO_n}(M) \times P_{S^1}(M)
	\end{split}
\end{displaymath}
be the corresponding $\Spc$-structures. And let the cocycles associated to this structures be, respectively, $\tilde{g}_{\alpha \beta}$, $\left. \tilde{g}_{\alpha \beta}\right|_M  $ and $\tilde{g}_{\alpha \beta}^1$. If we define the functions $\tilde{g}_{\alpha \beta}^2$ by
\begin{displaymath}
	\tilde{g}_{\alpha \beta}^1 \tilde{g}_{\alpha \beta}^2 = \left. \tilde{g}_{\alpha \beta} \right|_M
\end{displaymath}
it is easy to see, using an adapted frame, that the two sets of functions $\tilde{g}_{\alpha \beta}^1$ and $\tilde{g}_{\alpha \beta}^2$ commutes. This implies that $\tilde{g}_{\alpha \beta}^2$ satisfies the cocycle condition, because both $\tilde{g}_{\alpha \beta}$ and $\tilde{g}_{\alpha \beta}^1$ satisfies. The cocycles $\tilde{g}_{\alpha \beta}^2$ are exactly the $\Spc$-structure for the normal bundle $\nu(M)$. With this construction, if $L$, $L_1$ and $L_2$ denotes, respectively, the determinant bundle of the $\Spc$-structure of $Q$, $M$ and $\nu(M)$ we have the relation
\begin{displaymath}
	L = L_1 \otimes L_2
\end{displaymath}


Knowing that $\nu(M)$ has a natural $\Spc$-structure we can use the left regular representation of $\Ccl(n)$ on itself to construct the following $\Spc$-Clifford bundle (this bundles will act as spinor bundles)
\begin{equation}
	\begin{split}
		\Sigma^{\mathbb{C}}Q &:= P_{Spin^{\mathbb{C}}_{(n+m)}(Q)} 
			\times_{\rho_{(n+m)}} \mathbb{C}l_{(n+m)}, \\
		\left. \Sigma^{\mathbb{C}}Q \right|_M &:= \left.  P_{Spin^{\mathbb{C}}_{(n+m)}(Q)} \right|_M
		\times_{\rho_{(n+m)}} \mathbb{C}l_{(n+m)}, \\
		\Sigma^{\mathbb{C}} M &:= P_{Spin^{\mathbb{C}}_n(M)} 
			\times _{\rho_{n}}\mathbb{C}l_{(n)}, \\
		\Sigma^{\mathbb{C}} \nu(M) &:=P_{Spin^{\mathbb{C}}_m \nu(M)} 
		\times _{\rho _{m}} \mathbb{C}l_{(m)}.
	\end{split}
\end{equation}

Using the isomophism $\mathbb{C}l_n \hat{\otimes} \mathbb{C}l_m \simeq \mathbb{C}l_{(n+m)}$ and standard arguments, \cite{bar98}, we get the relation

\begin{equation}
	\Sigma^{\mathbb{C}} Q \mid_M \simeq \Sigma^{\mathbb{C}} M \hat{\otimes} \Sigma^{\mathbb{C}} \nu(M) =: \Sigma^{ad \mathbb{C}}.
\end{equation}

Let $\nabla ^{\Sigma ^{\mathbb{C}}Q},\nabla ^{\Sigma ^{\mathbb{C}}M}$ and $\nabla ^{\Sigma ^{\mathbb{C}}\nu}$ be the connection on $\sum^{\mathbb{C}}Q,\sum^{\mathbb{C}}M$ and $\sum^{\mathbb{C}}\nu(M)$ respectively, induced by the  Levi-Civita connections of $P_{SO_{(n+m)}}(Q)$, $P_{SO_{(n)}}(M)$, and $P_{SO_{(m)}}(\nu)$. We denote the connection on $\sum\nolimits^{ad\mathbb{C}}$ by

\begin{equation}
\nabla ^{\Sigma ^{ad\mathbb{C}}} = \nabla ^{\Sigma ^{\mathbb{C}}M\otimes \Sigma^{\mathbb{C}} \nu}:=\nabla
^{\Sigma ^{\mathbb{C}}M}\otimes Id+Id\otimes \nabla ^{\Sigma ^{\mathbb{C}}\nu}.
\end{equation}

The connections on these bundle are linked by the following Gauss formula:


\begin{equation} \label{gaussformula}
\nabla_{X}^{\Sigma^{\mathbb{C}} Q} \varphi =  \nabla_{X}^{\Sigma^{ad \mathbb{C}} } \varphi + \frac{1}{2} \sum _{i=1}^{n} e_i \cdot B(e_i,X) \cdot \varphi ,
\end{equation}
where $B:TM \times TM \rightarrow \nu(M)$ is the second fundamental form and $\left\{ e_1 \cdots e_n \right\}$ is a local orthonormal frame of $TM$. Here ``$\cdot $'' is the Clifford multiplication on $\Sigma^{\mathbb{C}}Q$.

Note that if we have a parallel spinor $\varphi$ in $\Sigma ^{\mathbb{C}}Q$, for exemple if $Q=\mathbb{R}^{n+m}$, then Eq.(\ref{gaussformula}) implies the following generalized Killing equation
\begin{equation}
    \nabla_{X}^{\Sigma^{ad \mathbb{C}} } \varphi =- \frac{1}{2} \sum _{i=1}^{n} e_i \cdot B(e_i,X) \cdot \varphi .
\end{equation}

\section{Constructing the Immersion}


To construct the immersion we need two steps. First we need to construct a vector valued inner product using the Clifford algebra structure of the $\Spc$-Clifford bundle. This first step does not change when we consider $\Spc$-structures instead of $\Sp$-structures. Therefore we just remember the construction by Bayard, Lawn and Roth \cite{bayard16} in the first subsection.

Second, we need to understand a Gauss type equation on the manifold. For this step the connection on the determinant bundle of the $\Spc$-structures is used and we show how the equations can be reformulated to this case. This is the principal part of the proof and is done on subsection \ref{immersion}.


\subsection{A $\mathbb{C}l_{(n+m)}$-valued inner product}

To make the converse, obtaining an immersion using spinors that satisfies certain equations, we need the following $\Ccl(n+m)$-valued inner product

\begin{eqnarray}
		\tau  :\Ccl_{(n+m)}& \rightarrow &\Ccl_{(n+m)} \\
		\tau (a~e_{i_{1}}e_{i_{2}}\cdots e_{i_{k}}) &:=&(-1)^{k}\bar{a}
			~e_{i_{k}}\cdots e_{i_{2}}e_{i_{1}}, \\
		\tau (\xi) &:=& \overline{\xi}	
	\end{eqnarray}

\begin{equation}
	\begin{split}
		\left\langle \left\langle \cdot ,\cdot \right\rangle \right\rangle:
			\mathbb{C}l_{(n+m)}\times \mathbb{C}l_{(n+m)} &\rightarrow \mathbb{C}l_{(n+m)} \\
		(\xi _{1},\xi _{2}) &\mapsto \left\langle \left\langle \xi _{1}, 
			\xi_{2}\right\rangle \right\rangle =\tau (\xi _{2})\xi _{1}.
	\end{split} \label{product}
\end{equation}

\begin{equation}
	\begin{split}
	\left\langle \left\langle (g\otimes s)\xi _{1},(g\otimes s) 
		\xi_{2}\right\rangle \right\rangle  &=s\overline{s}\tau 
		(\xi _{2})\tau (g)g \xi_{1}=\tau (\xi _{2})\xi _{1} = 
		\left\langle \left\langle \xi _{1}, \xi_{2}\right\rangle 
		\right\rangle , \\
	g\otimes s &\in Spin_{(n+m)}^{\mathbb{C}}\subset \mathbb{C}l_{(n+m)},
	\end{split}
\end{equation}
so the product is well defined on the $\Spc$-Clifford bundles, i.e., Eq.(\ref{product}) induces a $\mathbb{C}l_{(n+m)}$-valued map:
\begin{gather*}
\sum\nolimits^{\mathbb{C}}Q\times \sum\nolimits^{\mathbb{C}}Q\rightarrow 
\mathbb{C}l_{(n+m)} \\
(\varphi _{1},\varphi _{2})=([p,[\varphi _{1}]],[p,[\varphi _{2}]])\mapsto
\left\langle \left\langle \lbrack \varphi _{1}],[\varphi _{2}]\right\rangle
\right\rangle =\tau ([\varphi _{2}])[\varphi _{1}],
\end{gather*}
where $[\varphi _{1}]$, $[\varphi _{2}]$ are the representative of $\varphi_{1},\varphi _{2}$ in the $Spin^{\mathbb{C}}(n+m)$ frame $p\in P_{Spin^{\mathbb{C}}(n+m)}.$

\begin{lemma}
	The connection $\nabla ^{\Sigma ^{\mathbb{C}}Q}$ is compatible with the	product $\left\langle \left\langle \cdot ,\cdot \right\rangle \right\rangle .$
\end{lemma}

\begin{proof}
	Fix $s=(e_{1},...,e_{(n+m)}):U\subset M\subset Q\rightarrow P_{SO(n+m)}$ a local section of the frame bundle, $l:U\subset M\subset Q\rightarrow P_{s^{1}}$ a local section of the associated $S^1$-principal bundle, $w^{Q}:T(P_{SO(n+m)})\rightarrow so(n+m)$ is the Levi-Civita connection of $P_{SO(n+m)}$ and $iA:TP_{S^{1}}\rightarrow i\mathbb{R}$ is an arbitrary connection on $P_{S^{1}}$, denote by $w^{Q}(ds(X))=(w_{ij}(X))\in so(n+m),$ $iA(dl(X))=iA^{l}(X).$
	
	If $\psi =[p,[\psi ]]\ $and $\psi ^{\prime }=[p,[\psi ^{\prime }]]$ are	sections of $\sum\nolimits^{\mathbb{C}}Q$ we have:
	\begin{eqnarray*}
		\nabla _{X}^{\Sigma ^{\mathbb{C}}Q}\psi &=&\left[ p,X([\psi ])+\frac{1}{2} \sum\nolimits_{i<j}w_{ij}(X)e_{i}e_{j}\cdot \lbrack \psi ]+\frac{1}{2} iA^{l}(X)[\psi ]\right] , \\
		\left\langle \left\langle \nabla _{X}^{\Sigma ^{\mathbb{C}}Q}\psi ,\psi^{\prime }\right\rangle \right\rangle &=&\overline{[\psi ^{\prime }]}\left( X([\psi ])+\frac{1}{2}\sum\nolimits_{i<j}w_{ij}e_{i}e_{j}\cdot \lbrack \psi]+\frac{1}{2}iA^{l}(X)[\psi ]\right) , \\
		\left\langle \left\langle \psi ,\nabla _{X}^{\Sigma ^{\mathbb{C}}Q}\psi' \right\rangle \right\rangle &=&\overline{\left( X([\psi'])+\frac{1}{2}\sum\nolimits_{i<j}w_{ij}e_{i}e_{j}[\psi ']+\frac{1}{2}A^{l}[\psi '] \right) }[\psi ] \\
		&=&\left( X(\overline{[\psi ']})+\frac{1}{2}\sum \nolimits_{i<j}w_{ij}\overline{e_{i}e_{j}[\psi ']}+\frac{1}{2}	\overline{A^{l}}\overline{[\psi ']}\right) [\psi ] \\
		&=&\left( X(\overline{[\psi ']})-\frac{1}{2}\sum \nolimits_{i<j}w_{ij}\overline{[\psi ']}e_{i}e_{j}-\frac{1}{2}A^{l} \overline{[\psi ']}\right) [\psi ],
	\end{eqnarray*}
	then
	\begin{eqnarray*}
		\left\langle \left\langle \nabla _{X}^{\Sigma ^{\mathbb{C}}Q}\psi ,\psi^{\prime }\right\rangle \right\rangle +\left\langle \left\langle \psi,\nabla _{X}^{\Sigma ^{\mathbb{C}}Q}\psi ^{\prime }\right\rangle\right\rangle &=&\overline{[\psi ^{\prime }]}X(\xi )+X(\overline{[\psi^{\prime }]})[\psi ], \\
		X\left\langle \left\langle \psi ,\psi ^{\prime }\right\rangle \right\rangle&=&X\left( \overline{\xi ^{\prime }}\xi \right) =X(\overline{\xi ^{\prime }})\xi +\overline{\xi ^{\prime }}X(\xi ).
	\end{eqnarray*}
\end{proof}

\begin{lemma}
	The map $\left\langle \left\langle \cdot ,\cdot \right\rangle \right\rangle:\sum\nolimits^{\mathbb{C}}Q\times \sum\nolimits^{\mathbb{C}}Q\rightarrow \mathbb{C}l_{(n+m)}$ satisfies:
	
	\begin{enumerate}
		\item $\left\langle \left\langle X\cdot \psi ,\varphi \right\rangle
		\right\rangle =-\left\langle \left\langle \psi ,X\cdot \varphi \right\rangle
		\right\rangle ,~\psi ,\varphi \in \sum\nolimits^{\mathbb{C}}Q,~X\in TQ.$
		
		\item $\tau \left\langle \left\langle \psi ,\varphi \right\rangle
		\right\rangle =\left\langle \left\langle \varphi ,\psi \right\rangle
		\right\rangle ,~\psi ,\varphi \in \sum\nolimits^{\mathbb{C}}Q$
	\end{enumerate}
	
	\begin{proof}This is an easy calculation:
		\begin{enumerate} 
		
			\item $\left\langle \left\langle X\cdot \psi ,\varphi \right\rangle
			\right\rangle =\tau \lbrack \varphi ][X\cdot \psi ]=\tau \lbrack \varphi
			][X][\psi ]=-\tau \lbrack \varphi ]\tau \lbrack X][\psi ]=\left\langle
			\left\langle \psi ,X\cdot \varphi \right\rangle \right\rangle $
			
			\item $\tau \left\langle \left\langle \psi ,\varphi \right\rangle
			\right\rangle =\tau (\tau \lbrack \varphi ][\psi ])=\tau \lbrack \psi
			][\varphi ]=\left\langle \left\langle \varphi ,\psi \right\rangle
			\right\rangle .$
		\end{enumerate}
	\end{proof}
\end{lemma}

Note the same idea, product and properties are valid for the bundles $\sum^{%
	\mathbb{C}}Q,$ $\sum^{\mathbb{C}}M$, $\sum^{\mathbb{C}} \nu(M)$, $\sum\nolimits^{%
	\mathbb{C}}M\hat{\otimes}\sum\nolimits^{\mathbb{C}} \nu(M).$

\subsection{Spinorial Representation of Submanifolds in $\mathbb{R}^{n+m}$} \label{immersion}

Let $M$ a $n$-dimensional manifold, $E\rightarrow M$ a real vector bundle of rank $m$, assume that $TM$ and $N$ are oriented and $Spin^{\mathbb{C}}.$ Denote by $P_{SO_{n}}(M)$ the frame bundle of $TM$ and by $P_{SO_{m}}(E)$ the frame bundle of $E.$ The respective $Spin^{\mathbb{C}}$ structures are defined as%
\begin{eqnarray*}
	\Lambda ^{1\mathbb{C}} &:&P_{Spin_{n}^{\mathbb{C}}}(M)\rightarrow
	P_{SO_n}(M)\times P_{S^{1}}(M), \\
	\Lambda ^{2\mathbb{C}} &:&P_{Spin_{m}^{\mathbb{C}}}(E)\rightarrow
	P_{SO_m}(E)\times P_{S^{1}}(E).
\end{eqnarray*}%

We can define the bundle $P_{S^{1}}$ as the one with transition functions defined by product of transition functions of $P_{S^{1}}(M)$ and $P_{S^{1}}(E)$. It is not diffiult to see that there is a canonical bundle morphism: $\Phi :P_{S^{1}}(M)\times _{M}P_{S^{1}}(E)\rightarrow P_{S^{1}}$ such that, in any local trivialization, the following diagram comute:
\begin{center}
	\begin{equation*}
	\xymatrixcolsep{1pc}\xymatrixrowsep{1pc}\xymatrix{
		P_{S^{1}}(M)\times _{M}P_{S^{1}}(E) \ar[rr]^{~~~~~~~~~~ \Phi} \ar[dd] & &  P_{S^{1}}  \ar[dd] \\	& \\
		U_{\alpha} \times S^1 \times S^1 \ar[rr]^{~~~~~\phi_\alpha} & & U_{\alpha} \times S^1	}
	\end{equation*}
\end{center}
where $\phi_\alpha(x,r,s)=(x,rs), x\in U_{\alpha}, r,s \in S^1.$

Fix the following notation
\begin{eqnarray*}
	\sum\nolimits^{ad\mathbb{C}} &:&=\sum\nolimits^{\mathbb{C}}M\otimes
	\sum\nolimits^{\mathbb{C}}E\simeq \left( P_{Spin^{\mathbb{C}}(n)}\times
	_{M}P_{Spin^{\mathbb{C}}(m)}\right) \times \mathbb{C}l_{(n+m)}, \\
	N\sum\nolimits^{ad\mathbb{C}} &:&=\left( P_{Spin^{\mathbb{C}}(n)}\times
	_{M}P_{Spin^{\mathbb{C}}(m)}\right) \times Spin_{(n+m)}^{\mathbb{C}}.
\end{eqnarray*}

Here $iA^{1}:TP_{S^{1}}(M)\rightarrow i\mathbb{R}$, $iA^{2}:TP_{S^{1}}(E)\rightarrow i\mathbb{R}$ are arbitrary connections in $P_{S^{1}}(M)$ and $P_{S^{1}}(E)$. Denote a local section by $s=(e_{1},\cdots,e_{n}):U\rightarrow P_{SO_{n}}(M)$, $l_{1}:U\rightarrow P_{S^{1}}(M)$, $l_{2}:U\rightarrow P_{S^{1}}(E),$ $l=\Phi (l_{1},l_{2}):U\rightarrow P_{S^{1}}$. Now $iA:TP_{S^{1}}\rightarrow i\mathbb{R}$ is the connection
defined by $iA(d\Phi (l_{1},l_{2}))=iA_{1}(dl_{1})+iA_{2}(dl_{2}).$
Established this notation we have the following:

\begin{theorem}
	Let $M$ a simply connected $n$-dimensional manifold, $E\rightarrow M$ a
	vector bundle of rank $m$, assume that $TM$ and $E$ are oriented and $Spin^{%
		\mathbb{C}}.$ Suppose that $B:TM\times TM\rightarrow E$ is symmetric and
	bilinear. The following are equivalent:
	
	\begin{enumerate}
		\item There exist a section $\varphi \in \Gamma (N\sum\nolimits^{ad\mathbb{C%
		}})$ such that 
		\begin{equation}
		\nabla _{X}^{\Sigma ^{ad\mathbb{C}}}\varphi =-\frac{1}{2}\sum_{i=1}e_{i}%
		\cdot B(X,e_{i})\cdot \varphi +\frac{1}{2}i~A^{l}(X)\cdot \varphi ,~~\forall
		X\in TM.  \label{killing1}
		\end{equation}
		
		\item There exist an isometric immersion $F:M\rightarrow \mathbb{R}^{\left(
			n+m\right) }$ with normal bundle $E$ and second fundamental form $B$.
	\end{enumerate}
	
	Furthermore, $F=\int \xi $ where $\xi $ is the $\mathbb{R}^{(n+m)}$-valued $%
	1 $-form defined by%
	\begin{equation}
	\xi (X):=\left\langle \left\langle X\cdot \varphi ,\varphi \right\rangle
	\right\rangle ,~~\forall X\in TM.  \label{form}
	\end{equation}
\end{theorem}

\begin{proof}
	$2)\Rightarrow 1)$ Since $\mathbb{R}^{n+m}$ is contratible there exists a
	global section $s:\mathbb{R}^{n+m}\rightarrow P_{Spin^{\mathbb{C}}(n+m)}$,
	with a corresponding parallel orthonormal basis $h=(E_{1},\cdots ,E_{n+m}):	\mathbb{R}^{n+m}\rightarrow P_{SO(n+m)},$ and $l:\mathbb{R}^{n+m}\rightarrow P_{S^{1}},$ $\Lambda^{\mathbb{R}^{n+m}}(s)=(h,l).$ Fix a constant $[\varphi ]\in Spin^{\mathbb{C}}(n+m)\subset \mathbb{C}l_{(n+m)}$ and define
	the spinor field $\varphi =[s,[\varphi ]]\in \sum\nolimits^{\mathbb{C}}
	\mathbb{R}^{n+m}:=P_{Spin^{\mathbb{C}}(n+m)}\times \mathbb{C}l_{(n+m)},$
	again denote $w^{Q}(dh(X))=(w_{ij}^{h}(X))\in so(n+m),$ $iA(dl(X))=iA^{l}(X)\in i\mathbb{R},$
	\begin{eqnarray}
	\nabla _{X}^{\Sigma ^{\mathbb{C}}Q}\varphi &=&\left[ s,X([\varphi ])+\left\{ 
	\frac{1}{2}\sum\nolimits_{i<j}w_{ij}^{h}(X)E_{i}E_{j}+\frac{1}{2}%
	i~A^{l}(X)\right\} \cdot \lbrack \varphi ]\right] \notag
	\\
	&=&\left[ s,\frac{1}{2}i~A^{l}(X)\cdot \lbrack \varphi ]\right]. \notag \\
    &=&\frac{1}{2}i~A^{l}(X)\cdot \varphi	\label{parallel} 
	\end{eqnarray}

	Finally, restricting $\varphi$ to $\Sigma^{ad\mathbb{C}}$ and applying the gauss formula Eq.(\ref{gaussformula})%
	\begin{eqnarray}
	\nabla _{X}^{\Sigma ^{\mathbb{C}}Q}\varphi -\nabla _{X}^{\Sigma ^{ad\mathbb{C%
	}}}\varphi &=&\frac{1}{2}\sum_{i=1}e_{i}\cdot B(X,e_{i})\cdot \varphi  \notag
	\\
	\frac{1}{2}i~A^{l}(X)\cdot \varphi -\nabla _{X}^{\Sigma ^{ad\mathbb{C}%
	}}\varphi &=&\frac{1}{2}\sum_{i=1}e_{i}\cdot B(X,e_{i})\cdot \varphi  \notag
	\\
	\nabla _{X}^{\Sigma ^{ad\mathbb{C}}}\varphi &=&-\frac{1}{2}%
	\sum_{i=1}e_{i}\cdot B(X,e_{i})\cdot \varphi +\frac{1}{2}i~A^{l}(X)\cdot
	\varphi .  \label{killing}
	\end{eqnarray}

	\bigskip
	
	$1)\Rightarrow 2)$ The ideia here is to prove that the $1$-form $\xi $ Eq.(%
	\ref{form}) gives us an immersion preserving the metric, the second
	fundamental form and the normal connection. For this purpose, we will
	present the following lemmas:
	
	\begin{lemma}
		Suppose that $\varphi \in \Gamma (N\sum\nolimits^{ad\mathbb{C}})$ satisfies
		Eq.(\ref{killing1}) and define $\xi $ by Eq.(\ref{form}), then
		
		\begin{enumerate}
			\item $\xi $ is $\mathbb{R}^{(n+m)}$-valued $1$-form.
			
			\item $\xi $ is a closed $1$-form, $d\xi =0$
		\end{enumerate}
		
		\begin{proof}
			\begin{enumerate}
				\item If $\varphi =[p,[\varphi ]],X=[p,[X]],$ where $[\varphi ]$ and $[X]$
				represent $\varphi $ and $X$ in a given frame $\tilde{s}\in P_{Spin^{\mathbb{%
							C}}(n)}\times P_{Spin^{\mathbb{C}}(m)},$%
				\begin{equation*}
				\xi (X):=\tau \lbrack \varphi ][X][\varphi ]\in \mathbb{R}^{n}\subset
				Cl_{n}\subset \mathbb{C}l_{n},\text{ because }[\varphi ]\in Spin^{\mathbb{C}%
				}.
				\end{equation*}
				
				\item Supouse that in the point $x_{0}\in M$ $\nabla ^{M}X=\nabla ^{M}Y=0,$
				to simplify write $\nabla _{X}^{\Sigma ^{ad\mathbb{C}}}\varphi =\nabla
				_{X}\varphi $ and $\nabla ^{M}X=\nabla X$,%
				\begin{eqnarray*}
					X(\xi (Y)) &=&\left\langle \left\langle Y\cdot \nabla _{X}\varphi ,\varphi
					\right\rangle \right\rangle +\left\langle \left\langle Y\cdot \varphi
					,\nabla _{X}\varphi \right\rangle \right\rangle =(id-\tau )\left\langle
					\left\langle Y\cdot \varphi ,\nabla _{X}\varphi \right\rangle \right\rangle
					\\
					&=&(id-\tau )\left\langle \left\langle \varphi ,\frac{1}{2}%
					\sum\nolimits_{j=1}^{m}Y\cdot e_{j}\cdot B(X,e_{j})\cdot \varphi -\frac{1}{2%
					}A^{l}(X)iY\cdot \varphi \right\rangle \right\rangle , \\
					Y(\xi (X)) &=&(id-\tau )\left\langle \left\langle \varphi ,\frac{1}{2}%
					\sum\nolimits_{j=1}^{m}X\cdot e_{j}\cdot B(Y,e_{j})\cdot \varphi -\frac{1}{2%
					}A^{l}(Y)iX\cdot \varphi \right\rangle \right\rangle ,
				\end{eqnarray*}%
				from now on
				\begin{eqnarray}
				d\xi (X,Y) &=&X(\xi (Y))-Y(\xi (X)) \notag \\
				&=&(id-\tau )\left\langle \left\langle \varphi ,\frac{1}{2}%
				\sum\nolimits_{j=1}^{m}\left[ Y\cdot e_{j}\cdot B(X,e_{j})-X\cdot
				e_{j}\cdot B(Y,e_{j})\right] \cdot \varphi \right. \right. \notag \\ 
				&&\left. \left. +\frac{1}{2}i\left(
				A^{l}(Y)X-A^{l}(X)Y\right) \cdot \varphi \right\rangle \right\rangle \notag \\
				&=&(id-\tau )\left\langle \left\langle \varphi ,C\cdot \varphi \right\rangle
				\right\rangle,
				\end{eqnarray}
with $C=\frac{1}{2}\sum_{j=1}^{m}\left[ Y\cdot e_{j}\cdot B(X,e_{j})-X\cdot
				e_{j}\cdot B(Y,e_{j})\right] +\frac{1}{2}A^{l}(Y)iX-\frac{1}{2}A^{l}(X)iY.$
				Write $X=\sum_{k=1}^{m}x^{k}e_{k};~Y=\sum_{k=1}^{m}y^{k}e_{k}$ then
				\begin{eqnarray*}
				\sum\nolimits_{k=1}^{m}X\cdot e_{k}\cdot B(Y,e_{k})
				&=&\sum\nolimits_{j=1}^{m}\sum\nolimits_{k=1}^{m}x^{k}e_{k}\cdot
				e_{j}\cdot B(Y,e_{j}) \notag \\ &=&-B(Y,X)+\sum\nolimits_{j=1}^{m}\sum\nolimits 
				_{\substack{ k=1  \\ k\neq j}}^{m}x^{k}e_{k}\cdot e_{j}\cdot B(Y,e_{j}) \\
				\sum\nolimits_{l=1}^{m}Y\cdot e_{k}\cdot B(X,e_{k})
				&=&\sum\nolimits_{j=1}^{m}\sum\nolimits_{k=1}^{m}y^{k}e_{k}\cdot
				e_{j}\cdot B(X,e_{j}) \notag \\
				&=&-B(X,Y)+\sum\nolimits_{j=1}^{m}\sum\nolimits 
				_{\substack{ k=1  \\ k\neq j}}^{m}y^{k}e_{k}\cdot e_{j}\cdot B(X,e_{j}) 
				\end{eqnarray*}
				from what
				\begin{eqnarray}
				C &=&\frac{1}{2}\left[ \sum\nolimits_{j=1}^{m}\sum\nolimits_{\substack{ 
						k=1  \\ k\neq j}}^{m}e_{k}\cdot e_{j}\cdot \left[
				y^{k}B(X,e_{j})-x^{k}B(Y,e_{j})\right] \right. \notag \\ && \left. +i(A^{l}(Y)X-A^{l}(X)Y)\right] \notag \\
				\tau ([C]) &=&-\frac{1}{2}\left[ \sum\nolimits_{j=1}^{m}\sum\nolimits 
				_{\substack{ k=1  \\ k\neq j}}^{m}\left[ y^{k}B(X,e_{j})-x^{k}B(Y,e_{j})	\right] \right] \cdot e_{j}\cdot e_{k}  \notag \\ && + \frac{i}{2}(A^{l}(Y)\left[ X\right]
				-A^{l}(X)\left[ Y\right] ) \notag \\
				&=&\frac{1}{2}\left[ \sum\nolimits_{j=1}^{m}\sum\nolimits_{\substack{ k=1 
						\\ k\neq j}}^{m}e_{k}\cdot e_{j}\cdot \left[ y^{k}B(X,e_{j})-x^{k}B(Y,e_{j})%
				\right] \right] \notag \\ 
				&& + \frac{i}{2}(A^{l}(Y)\left[ X\right] -A^{l}(X)\left[ Y\right]
				)=[C].
				\end{eqnarray}
				Which implies that%
				\begin{equation*}
				d\xi (X,Y)=(id-\tau )\left\langle \left\langle \varphi ,C\cdot \varphi
				\right\rangle \right\rangle =(id-\tau )(\tau \lbrack \varphi ]\tau \lbrack
				C][\varphi ])=0.
				\end{equation*}
			\end{enumerate}
		\end{proof}
	\end{lemma}
	
	From the fact that $M$ is simply connected and $\xi $ is closed, from the
	Poincar\'{e}'s Lemma we know that there exists a%
	\begin{equation*}
	F:M\rightarrow \mathbb{R}^{\left( n+m\right) }
	\end{equation*}%
	such that $dF=\xi .$ The next lemma allows us to conclude the proof of the
	theorem.
	
	\begin{lemma}
		\begin{enumerate}
			\item The map $F:M\rightarrow \mathbb{R}^{n},$ is an isometry.
			
			\item The map%
			\begin{eqnarray*}
				\Phi _{E} &:&E\rightarrow M\times \mathbb{R}^{n} \\
				X &\in &E_{m}\mapsto (F(m),\xi (X))
			\end{eqnarray*}%
			is an isometry between $E$ and the normal bundle of $F(M)$ into $\mathbb{R}%
			^{\left( n+m\right) },$ preserving connections and second fundamental forms.
		\end{enumerate}
		
		\begin{proof}
			\begin{enumerate}
				\item Let $X,Y\in \Gamma (TM\oplus E),$ consequently%
				\begin{eqnarray}
				\left\langle \xi (X),\xi (Y)\right\rangle &=&-\frac{1}{2}\left( \xi (X)\xi
				(Y)-\xi (Y)\xi (X)\right) \notag \\ &=&-\frac{1}{2}\left( \tau \lbrack \varphi
				][X][\varphi ]\tau \lbrack \varphi ][Y][\varphi ]-\tau \lbrack \varphi
				][Y][\varphi ]\tau \lbrack \varphi ][X][\varphi ]\right) \notag \\
				&=&-\frac{1}{2}\tau \lbrack \varphi ]\left( [X][Y]-[Y][X]\right) [\varphi
				]=\tau \lbrack \varphi ]\left( \left\langle X,Y\right\rangle \right)
				[\varphi ] \notag \\
				&=&\left\langle X,Y\right\rangle \tau \lbrack \varphi ][\varphi
				]=\left\langle X,Y\right\rangle .
				\end{eqnarray}
				This implies that $F$ is an isometry, and that $\Phi _{E}$ is a bundle map
				between $E$ and the normal bundle of $F(M)$ into $\mathbb{R}^{n}$ which
				preserves the metrics of the fibers.
				
				\item Denote by $B_{F}$ and $\nabla ^{\prime F}$ the second fundamental form
				and the normal connection of the immersion $F$. We want to show that:%
				\begin{eqnarray*}
					i)\xi (B(X,Y)) &=&B_{F}(\xi (X),\xi (Y)), \\
					ii)\xi (\nabla _{X}^{\prime }\eta ) &=&(\nabla _{\xi (X)}^{\prime F}\xi
					(\eta )),
				\end{eqnarray*}%
				for all $X,Y\in \Gamma (TM)$ and $\eta \in \Gamma (E)$.
				
				$i)$ First note that:%
				\begin{equation*}
				B^{F}(\xi (X),\xi (Y)):=\{\nabla _{\xi (X)}^{F}\xi (Y)\}^{\bot }=\{X(\xi
				(Y))\}^{\bot },
				\end{equation*}%
				where the superscript $\bot $ means that we consider the component of the
				vector which is normal to the immersion. We know that
				\begin{eqnarray}
				X(\xi (Y))&=&(id-\tau )\left\langle \left\langle \varphi ,\frac{1}{2}%
				\sum\nolimits_{j=1}^{m}Y\cdot e_{j}\cdot B(X,e_{j})\cdot \varphi -\frac{1}{2%
				}A^{l}(X)iY\cdot \varphi \right\rangle \right\rangle \notag \\
				&=&(id-\tau )\left\langle \left\langle \varphi ,\frac{1}{2}\left(
				\sum\nolimits_{j=1}^{m}\sum\nolimits_{k=1}^{m}y^{k}e_{k}\cdot e_{j}\cdot
				B(X,e_{j})  \right. \right. \right. \notag \\
				&&  -A^{l}(X)iY \Big) \cdot \varphi \bigg\rangle \bigg\rangle \notag \\
				&=&(id-\tau )\left\langle \left\langle \varphi ,\frac{1}{2}\left(
				\sum\nolimits_{j=1}^{m}y^{j}e_{j}\cdot e_{j}\cdot
				B(X,e_{j}) \right.\right. \right.  \notag \\
				& &\left.\left.\left. +\sum\nolimits_{j=1}^{m}\sum\nolimits_{k=1,k\neq j}^{m}y^{k}e_{k}\cdot e_{j}\cdot B(X,e_{j})-A^{l}(X)iY\right) \cdot \varphi
				\right\rangle \right\rangle \notag \\
				&=&(id-\tau )\left\langle \left\langle \varphi ,\frac{1}{2}\left(
				-B(X,Y)+D\right) \cdot \varphi \right\rangle \right\rangle ,
				\end{eqnarray}
				where%
				\begin{eqnarray*}
					D &=& \sum\nolimits_{j=1}^{m}\sum\nolimits_{k=1,k\neq j}^{m}y^{k}e_{k}\cdot
					e_{j}\cdot B(X,e_{j})-A^l(X)iY \\
					\tau \lbrack D] &=&[D].
				\end{eqnarray*}%
				Consequently%
				\begin{eqnarray*}
					X(\xi (Y)) &=&\frac{1}{2}(id-\tau )\left\langle \left\langle \varphi ,\left(
					-B(X,Y)+D\right) \cdot \varphi \right\rangle \right\rangle \\
					&=&-\tau \lbrack \varphi ]\tau \lbrack B(X,Y)][\varphi ]=\left\langle
					\left\langle \varphi ,B(X,Y)\cdot \varphi \right\rangle \right\rangle \\
					&=&\xi (B(X,Y)).
				\end{eqnarray*}%
				Therefore we conclude%
				\begin{eqnarray*}
					B_{F}(\xi (X),\xi (Y)) &:&=B^{F}(\xi (X),\xi (Y)):=\{\nabla _{\xi
						(X)}^{F}\xi (Y)\}^{\bot }=\{X(\xi (Y))\}^{\bot } \\
					&=&\{\xi (B(X,Y))\}^{\bot }=\xi (B(X,Y)),
				\end{eqnarray*}%
				here we used the fact that $F=\int \xi $ is an isometry: $B(X,Y)\in
				E\Rightarrow \xi (B(X,Y))\in TF(M)^{\bot }.$ Then $i)$ follows.
				
				$ii)$ First note that
				\begin{eqnarray*}
				\nabla _{\xi (X)}^{F}\xi (\eta )&=&\left\{ X(\xi (\eta ))\right\} ^{\bot
				}=\left\{ X\left\langle \left\langle \eta \cdot \varphi ,\varphi
				\right\rangle \right\rangle \right\} ^{\bot } \notag \\ &=& \left\langle \left\langle
				\nabla _{X}\eta \cdot \varphi ,\varphi \right\rangle \right\rangle ^{\bot
				}+\left\langle \left\langle \eta \cdot \nabla _{X}\varphi ,\varphi
				\right\rangle \right\rangle ^{\bot }+\left\langle \left\langle \eta \cdot
				\varphi ,\nabla _{X}\varphi \right\rangle \right\rangle ^{\bot }.
				\end{eqnarray*}%
				I will show that:%
				\begin{equation*}
				\left\langle \left\langle \eta \cdot \nabla _{X}\varphi ,\varphi
				\right\rangle \right\rangle ^{\bot }+\left\langle \left\langle \eta \cdot
				\varphi ,\nabla _{X}\varphi \right\rangle \right\rangle ^{\bot }=0.
				\end{equation*}%
				In fact
				\begin{eqnarray*}
				&&\left\langle \left\langle \eta \cdot \nabla _{X}\varphi ,\varphi
				\right\rangle \right\rangle +\left\langle \left\langle \eta \cdot \varphi
				,\nabla _{X}\varphi \right\rangle \right\rangle \notag \\
				&=&(id-\tau )\left\langle \left\langle \eta \cdot \nabla _{X}\varphi
				,\varphi \right\rangle \right\rangle \notag \\
				&=&(-id+\tau )\left\langle \left\langle \left[ \frac{1}{2}%
				\sum\nolimits_{j=1}^{m}\eta \cdot e_{j}\cdot B(X,e_{j})\cdot \varphi -\frac{%
					1}{2}A^{l}(X)i\eta \cdot \varphi \right] ,\varphi \right\rangle \right\rangle
				\notag \\
				&=&(-id+\tau )\left\langle \left\langle \left[ -\frac{1}{2}%
				\sum\nolimits_{j=1}^{m}\sum\nolimits_{p=1}^{n}\sum%
				\nolimits_{k=1}^{n}n^{p}b_{j}^{k}e_{j}\cdot f_{p}\cdot f_{k}-\frac{1}{2}%
				A^{l}(X)i\eta \right] \cdot \varphi ,\varphi \right\rangle \right\rangle \notag \\
				&=&(-id+\tau )\left\langle \left\langle \left[ \frac{1}{2}%
				\sum\nolimits_{j=1}^{m}\sum\nolimits_{p=1}^{n}n^{p}b_{j}^{p}e_{j} \right. \right. \right.  \notag \\
				&& \left. \left. \left. -\frac{1}{%
					2}\sum\nolimits_{j=1}^{m}\sum\nolimits_{p=1}^{n}\sum\nolimits_{k=1,k\neq
					p}^{n}n^{p}b_{j}^{k}e_{j}\cdot f_{l}\cdot f_{k}-\frac{1}{2}A^{l}(X)iN\right]
				\cdot \varphi ,\varphi \right\rangle \right\rangle ,
				\end{eqnarray*}
				from what%
				\begin{eqnarray*}
					&&\left\langle \left\langle N\cdot \nabla _{X}\varphi ,\varphi \right\rangle
					\right\rangle +\left\langle \left\langle N\cdot \varphi ,\nabla _{X}\varphi
					\right\rangle \right\rangle \\
					&=&\tau \lbrack \varphi ][\frac{1}{2}\sum_{j=1}^{n}%
					\sum_{l=1}^{m}n^{l}b_{j}^{l}e_{j}][\varphi ]+\tau \lbrack \varphi ][\frac{1}{%
						2}\sum_{j=1}^{n}\sum_{l=1}^{m}n^{l}b_{j}^{l}e_{j}][\varphi ] \\
					&=&\tau \lbrack \varphi
					][\sum_{j=1}^{n}\sum_{l=1}^{m}n^{l}b_{j}^{l}e_{j}][\varphi ]=\tau \lbrack
					\varphi ][V][\varphi ]=:\xi (V)\in TF(M) \\
					&\Rightarrow &\left\langle \left\langle \eta \cdot \nabla _{X}\varphi
					,\varphi \right\rangle \right\rangle ^{\bot }+\left\langle \left\langle \eta
					\cdot \varphi ,\nabla _{X}\varphi \right\rangle \right\rangle ^{\bot }=0.
				\end{eqnarray*}%
				In conclusion%
				\begin{equation*}
				\nabla _{\xi (X)}^{F}\xi (\eta )=\left\langle \left\langle \nabla _{X}\eta
				\cdot \varphi ,\varphi \right\rangle \right\rangle ^{\bot }=\left\langle
				\left\langle \nabla _{X}\eta \cdot \varphi ,\varphi \right\rangle
				\right\rangle ^{\bot }=\xi (\nabla _{X}\eta )^{\bot }=\xi (\nabla
				_{X}^{\prime }\eta ).
				\end{equation*}%
				At the end $ii)$ follows.
			\end{enumerate}
		
	\end{proof}
	
\end{lemma}
	
	With these Lemmas the theorem is proved.

\end{proof}

\end{document}